\documentclass[11pt,reqno,a4paper]{amsart}
 \usepackage{amsmath,amsrefs,amssymb,tensor, eufrak, dsfont}
\usepackage{graphicx}
\usepackage{xcolor}
\definecolor{darkblue}{HTML}{004C93} 
\definecolor{MainRed}{rgb}{.6, .1, .1}
\usepackage[colorlinks,citecolor=darkblue,urlcolor=blue,bookmarks=false,hypertexnames=true]{hyperref}
\usepackage{cleveref}
\usepackage{mathrsfs}

\usepackage{todonotes}

\definecolor{myblue}{RGB}{0, 102, 204}
\hypersetup{colorlinks=true,
	linkcolor=MainRed,
	citecolor=myblue,
	filecolor=myblue,
	urlcolor=myblue}

\numberwithin{equation}{section}

\DeclareMathOperator{\diver}{div}
\DeclareMathOperator{\loc}{loc}

\DeclareMathOperator{\bigO}{O}
\DeclareMathOperator{\smallo}{o}

\DeclareMathOperator{\crits}{2^\star(s)}

\newcommand{\R}{\mathbb{R}}

\newcommand{\N}{\mathbb{N}}

\newcommand{\eps}{\varepsilon}

\renewcommand{\[}{\left[}
\renewcommand{\]}{\right]}
\renewcommand{\(}{\left(}
\renewcommand{\)}{\right)}

\def \ra{r_{\alpha}}
\def \bb{\hbox}

\def \ma{\mu_{\alpha}}
\def \nna{\nu_{\alpha}}
\def \ua{u_{\alpha}}

\def \dg{\Delta_g}

\def \ha{\widehat{h}_{\alpha}}
\def \tua{\widetilde{u}_{\alpha}}

\def \la{\lim_{\alpha\to +\infty}}
\def \pr{\partial}

\def \uha  {\widehat{u}_{\alpha}}
\def \bns{\mathfrak{b}_{n,s}}

\def \bnss{\mathfrak{b}_{3,s}}
\def \bn{\mathfrak{b}_{4,s}}

\def \ut  {\widetilde{u}_{\alpha}}

\def \xo{x_0}
\def \ig{\iota_{g}}

\def \po{\varphi_{\infty}}
\def \sg{\operatorname{Scal}_g}
\def \tm{\mathcal{TM}}

\newtheorem{theorem}{Theorem}[section]
\newtheorem{prop}{Proposition}[section]
\newtheorem{lemma}{Lemma}[section]
\newtheorem{step}{Step}[section]
\newtheorem{claim}{Claim}[section]
\newtheorem{remark}{Remark}[section]
\newtheorem{corollary}{Corollary}[section]

\begin{document}

\title[Hardy-Sobolev Equation on manifolds]{Compactness for the Hardy-Sobolev equation on manifolds}
\author{Hussein Cheikh Ali}
\address{Laboratoire Paul Painlev\'e, Universit\'e de Lille, Cit\'e Scientifique - B\^atiment M2 59655 Villeneuve d'Ascq Cedex France}
\email{frederic.cheikh-ali@univ-lille.fr}
\author{Saikat Mazumdar}
\address{Saikat Mazumdar , Department of Mathematics, Indian Institute of Technology, Bombay, Powai, Mumbai, Maharashtra 400076, India}
\email{saikat.mazumdar@iitb.ac.in, saikat@math.iitb.ac.in}
\date\today
\thanks{The authors would like to sincerely thank Prof. Fr{\'e}d{\'e}ric Robert for suggesting this problem and for many helpful discussions. \\ The second author gratefully acknowledges the support from the MATRICS grant MTR/2022/000447 of the Science and Engineering Research Board (currently ANRF) of India.}
\begin{abstract}
Let $(M, g)$ be a closed Riemannian manifold of dimension $n \geq 3$, and let $h \in C^1(M)$ be such that the operator $\Delta_g + h$ is coercive. Fix  $x_0 \in M$ and $s \in (0, 2)$. We obtain uniform bounds on the solutions of the critical \emph{Hardy-Sobolev equation}:
\begin{equation}\label{HS0}
\tag{{\color{MainRed}HS}}	\left\{\begin{array}{ll}
\Delta_{g}u + hu = \frac{u^{\crits-1}}{d_{g}(\xo,x)^{s}}   & \hbox{ in }M\setminus\{\xo\},  \\
\qquad u > 0  &\hbox{ in }M\setminus\{\xo\},
\end{array}\right.
\end{equation}
where $\Delta_{g}:=-\diver_{g}(\nabla)$ and $\crits:=2(n-s)/(n-2)$.  More precisely, we assume  $h(x_0)<\frac{(n-2)(6-s)}{12(2n-2-s)}\mathrm{Scal}_g(x_0),$ when $n \geq 4$, and $h\le\frac{1}{8}\sg$, $h(\xo)<\frac{1}{8}\sg(\xo)$ when $n = 3$. Here, $\mathrm{Scal}_g$ denotes the scalar curvature of $(M, g)$. These conditions were introduced in \cite{HCA4}, and shown to be optimal in \cite{CAR} for a single bubble configuration when $n\ge7$ . 

\noindent
We do not assume any bounds on the energy or the Sobolev norm of the solutions.
\end{abstract}
\maketitle

\section{Introduction and main results}\label{Sec1}
The aim of this paper is to obtain compactness results for the critical Hardy-Sobolev on closed manifolds.

\noindent
Let $(M,g)$ be a closed and smooth Riemannian manifold of dimension $n \geq 3$.  We fix $\xo \in M$, $s\in (0,2)$, and  let  $\crits:= 2(n-s)/(n-2)$ denote the critical Hardy-Sobolev exponent. We consider the following \emph{Hardy-Sobolev equation}:
 \begin{align}\label{eqn:HS1}
	\left\{\begin{array}{ll}
	\Delta_{g}u + hu = \dfrac{u^{\crits-1}}{d_{g}(\xo,x)^{s}} & \hbox{ in }M\setminus\{\xo\},\\
		\quad u > 0  &\hbox{ in }M\setminus\{\xo\}.
	\end{array}\right.
\end{align}
Here $\Delta_{g}:=-\diver_{g}(\nabla)$  is the Laplace-Beltrami operator with (minus sign convention) and  $h \in C^{1}(M)$.  We  will also assume throughout this paper that $\Delta + h$ is coercive, that is, there exists $\mathscr{C}>0$ such that 
\begin{equation*}
\int_{M} \left(|\nabla v|_g^2+h v^2 \right) \, dv_g \geq \mathscr{C}\|v\|^{2}_{H^{2}_{1}}\, \bb{ for all } v\in H_1^2(M),
\end{equation*}
where the Sobolev space $H^{2}_{1}(M)$ can equivalently  be defined as the completion of $C^{\infty}(M)$ with respect to the norm $\|\cdot\|_{H^{2}_{1}}:=\sqrt{\|\nabla\cdot\|^{2}_{L^{2}}+\|\cdot\|^{2}_{L^{2}}}~$\,.
\smallskip

The exponent $\crits$ is critical for the Sobolev embedding $H^{2}_{1}(M)\to L^{p}(M,d_{g}(\cdot,x_{0})^{-s})$, where $1\le p\le \crits$ and our equation \eqref{eqn:HS1} is variationaly noncompact in general. Note that the singular term $d_g(\xo,x)^{-s}$ in equation \eqref{eqn:HS1} breaks the natural conformal invariance property of the problem. 

\noindent
We further let $\sg$  denote the scalar curvature of the manifold $(M,g)$ and   
\begin{align}\label{def:c}
	\mathfrak{c}_{n,s}:=\frac{(n-2)(6-s)}{12(2n-2-s)},~ 0\le s<2.
\end{align}

\noindent
The case $s=0$ corresponds to Yamabe-type equations, and likewise the potential $\mathfrak{c}_{n,s}\sg$ plays a similar defining role here in obtaining existence and compactness of solutions of \eqref{eqn:HS1}. For $s>0$ it turns out the behaviour around the singularity $x_{0}$ only really matters, as the problem is subcritical and hence compact on domains not containing  $x_{0}$. We exploit this phenomenon in this paper, and this further simplifies and shortens many of the arguments needed when $s=0$ for the case of Yamabe-type problems. We consider solutions that can a priori develop an arbitrary number of bubbles in the bubble tree decomposition. Moreover, we do not assume any bound on the energy or the Sobolev norm of the solutions. 
\smallskip

\noindent
The existence of solutions of \eqref{eqn:HS1} is well-understood, and in \cite{J1} Jaber showed the existence of positive extremals under the local sign condition $h(\xo)<\mathfrak{c}_{n,s}\sg(\xo)$ for  $n\geq 4$. For dimension $n=3$, as in the non-singular case $s=0$, the mass of the operator $\Delta_g +h$ plays a crucial role, and Jaber in the same paper gave an existence result when the mass at $x_{0}$ is positive. For the definition of mass, see \eqref{green0} and one may think of it as the constant term in the expansion of the Green's function of the operator $\Delta_{g}+h$.  Concerning regularity, it follows from Jaber \cite{J1} that any solution $u \in H^{2}_{1}(M)$ of \eqref{eqn:HS1}  satisfies $u \in C^{0,\theta}(M)\cap C^{2,\eta}_{\loc}(M \setminus \{\xo\})$, for all $0 <\theta < \min\{ 1, 2-s\} $ and $0< \eta <1$. Also see Jaber \cite{Jaber2} and Cheikh-Ali \cite{HCA4}  for a sharp form of the Hardy-Sobolev inequality on manifolds, and the conditions which give the existence of minimizers. \par

\smallskip Motivated by results on Yamabe-type problems, we obtain in this paper a compactness result for the solutions of the Hardy-Sobolev equation \eqref{eqn:HS1}. 
Compactness of solutions of \eqref{eqn:HS1} for the case $s=0$ is a very well-studied topic and the literature is abundant. For brevity, we refer the interested reader to Druet \cite{DruetIMRN} and the book by Hebey \cite{Hbooks} and the references therein. Compactness for the Hardy-Sobolev equation with $d_{g}(x_{0},x)$ replaced by the distance function from submanifolds was recently explored in \cite{HP}. For the geometric Yamabe poblem some notable developments are Schoen \cite{Schoen1, Schoen2}, Li-Zhu \cite{LiZhu} when $n=3$, Druet \cite{DruetIMRN}  when $n\leq5$,  Marques \cite{MarquesJDG} when $n \leq 7$,  Li-Zhang \cite{LiZhangCPDE,LiZhangJFA} when  $n \leq 11$ and Khuri-Marques-Schoen \cite{KMSJDG} when $n\leq 24$. Also see Druet-Premoselli \cite{DruetPremoselli} for stability of the Einstein–Lichnerowicz constraint system, and Premoselli-Vetois \cite{PV} for compactness of sign-changing solutions of the Yamabe type equation. For domains in $\R^{n}$ with Dirichlet boundary conditions and $s=0$ we refer to Druet and Laurain \cite{DruetLaurain}, Laurain and Konig \cite{LK}, and more recently to Cheikh-Ali and Premoselli \cite{CAPAPDE}. We also refer to Ghoussoub-Mazumdar-Robert \cite{GMRAMS, GMR2} for the blow up analysis of singular Hardy–Schr\"odinger type problems in the case $s \in (0,2)$. 
\smallskip

Our main result is the following theorem.
\begin{theorem}\label{main:theo1}
Let $(M,g)$ be a closed and smooth Riemannian manifold of dimension $n\geq 3$, fix $\xo \in M$ and $s\in (0,2)$. Let $h \in C^{1}(M)$ be such that the operator $\Delta_{g}+ h$ is coercive and assume that one of
the following assumptions is satisfied:
\begin{itemize}
	\item $n\geq 4$ and $h(\xo)<\mathfrak{c}_{n,s} \sg(\xo)$, where $c_{n,s}$ is given by \eqref{def:c}.
	\item $n=3$ and $m_h(\xo)\not=0$, where $m_h(\xo)$ is the mass of the operator $\Delta_g+h$ defined in \eqref{green0}.
\end{itemize}
Then, there exists a constant $\Lambda=\Lambda(M,g,s, h)$ such that any solution $u \in C^{0,\theta}(M)\cap C^{2}(M \setminus \{\xo\})$ of \eqref{eqn:HS1}, with $0<\theta<\min\{1, 2-s\}$ satisfies: 
\begin{align}\label{main bd}
\left\|u\right\|_{C^{0}(M)}+\left\|1/u\right\|_{C^{0}(M)}\le\Lambda.
\end{align}
\end{theorem}

Note, we do not assume any $H^{2}_{1}(M)$ bound on solutions. The condition on the scalar curvature is optimal and has appeared before in Jaber \cite{J1}, and in \cite{HCA4}, where the first author arrived at the same conclusion for a sequence close to the single bubble profile. Very recently, the authors in \cite{CAR} constructed a sequence of solutions of \eqref{eqn:HS1} blowing up with the profile of a bubble such that $h\to h_{0}$ in $C^{p}(M)$, $p\ge2$, where $h_{0}(x_{0})=\mathfrak{c}_{n,s}\sg(x_{0})$ for $n\ge7$. Also see \cite{chen} for the construction of a blowing up sequence with perturbations of the nonlinear power. 
\smallskip

\noindent
Our proof is based on a priori asymptotic analysis of sequences of blowing-up solutions $(u_{\alpha})_{\alpha\in\mathbb{N}}$ of \eqref{eqn:HS1}. Compared to the case $s=0$, the presence of the singularity creates many simplifications and we show that the blow up is highly localized at $x_{0}$ and can be controlled globally by the canonical bubble centred at $x_{0}$.  Then, using a Pohozaev identity we obtain a balancing condition contradicting the assumptions of our Theorem. We effectively reduce the analysis of multi-bubble blow-up solutions to the study of a single bubble-type configuration. It seems our Theorem \ref{main:theo1} is the first general compactness result for the Hardy-Sobolev equation on manifolds.
\smallskip

The so-called {\emph mass-function}  $m_h$ of $\Delta_g +h$ is defined as follows: let $G_{\xo}:M\backslash\{\xo\}\to \R$ be the Green's function of the coercive operator $\Delta_g +h$. For $n=3$ we have the following expansion around the pole at $x_{0}$
\begin{equation}\label{green0}
G_{h}(\xo,x)=\frac{1}{4 \pi d_g(\xo,x)}+\beta(x_0,x) \bb{ for all } x\in M\backslash \{\xo\},
\end{equation}
for some $\beta(\xo,\cdot) \in C^{2}(M\backslash\{\xo\})\cap C^{0,\vartheta}(M)$ where $\vartheta \in(0,1)$, and we then define $m_h(\xo) = \beta(\xo,\xo)$. For an expansion of the Green's function, see, for example, Appendix A of Druet, Hebey, and Robert \cite{DHR}, and Robert \cite{Rgreen}. By  the positive mass theorem of Schoen and Yau \cite{SY} one has $m_{\sg/8}(\xo)\ge0$ and  $m_{\sg/8}(\xo)=0$ if and only $(M,g)$ is conformally diffeomorphic to the unit $3$-dimensional sphere $\mathbb{S}^{3}$ with the spherical metric $g_{std}$. 
Then applying the maximum principle, Theorem \ref{main:theo1} implies the following as in the case $s=0$.



\begin{corollary}\label{main:coro}
Let $(M,g)$ be a closed and smooth Riemannian manifold of dimension $n=3$ with positive scalar curvature $\sg>0$ in $M$. Fix $\xo \in M$ and $s\in (0,2)$. Let $h \in C^{1}(M)$ be such that the operator $\Delta_{g}+ h$ is coercive and  $h\le \sg/8$ in $M$ but  $h(x_{0})<\sg(x_{0})/8$ Then, there exists a constant $\Lambda=\Lambda(M,g,s, h)$ such that any solution $u \in C^{0}(M)\cap C^{2}(M \setminus \{\xo\})$ 
of \eqref{eqn:HS1} satisfies the bound \eqref{main bd}. 
\end{corollary}


 \smallskip 

\noindent
This paper is organized as follows. In Section~\ref{PN}, we recall some preliminary results that will be used throughout the paper. The proofs of the main Theorem ~\ref{main:theo1} and Corollary \ref{main:coro} are given in Section~\ref{proofmaintheo}. In Section~\ref{controlblowup} we perform a blow-up analysis and establish Proposition~\ref{main:prop}. In particular, necessary conditions for the blow-up of the sequence $(u_\alpha)_{\alpha \in \mathbb{N}}$ are obtained via an appropriate Pohozaev identity.

All convergences are up to a subsequence, and $C$ will denote a generic constant that depends on $n$, $s$, and possibly $(M,g)$.

\section{Preliminaries and notations}\label{PN}

$(M,g)$ will denote a closed and smooth Riemannian manifold of dimension $n\geq 3$. The linear space $\R^n$ will be endowed with the canonical Euclidean metric. On the tangent bundle of $M$, the exponential map $\exp: \tm\to M$ is smooth and there exists an injectivity radius $r_0>0$ such that for any $\xo\in M$ the map
\begin{equation*}
	\exp_{\xo}|_{B(0,r_0)}:B(0,r_0)\mapsto B_{g}(\xo,r_0)
\end{equation*}
is a diffeomorphism. Here $B(0,r_0)$ denotes the ball of radius $r_0$ centered at $0$ in $\R^n$, and $B_g(\xo,r_0)=\{y\in M \bb{ such that } d_g(y,x_0)<r_0\}$ denotes the geodesic ball of radius $r_0$ centered at $\xo$. Also, $\ig:=i_g(M)$ will denote the injectivity radius of $(M,g)$ and $\Delta_{g}:=-\diver_{g}(\nabla)$  is the Laplace-Beltrami operator with minus sign convention. Recall
\begin{equation*}
	-\left( \Delta_{g}-\Delta_{\xi}\right)  = \left( g^{ij}-\delta^{ij}\right) \partial_{ij}- g^{ij}\Gamma_{ij}^k\partial_k,
\end{equation*} 
where $\Delta_{\xi}$ denotes the usual Euclidean laplacian, and  $\Gamma_{ij}^k$ are the Christoffel symbols and for all $i,j,k,m \in \{1,...,n\}$
\begin{equation*}
\Gamma_{ij}^k=\frac{1}{2}g^{km}\left( \partial_ig_{jm}+\partial_jg_{im}-\partial_mg_{ij}\right).
\end{equation*}
\smallskip

\noindent
Coming back to our Hardy-Sobolev equation,  we define for $\mu>0$
\begin{equation*}
	U_{\mu}(X):=\mu^{\frac{2-n}{2}}U\left( \mu^{-1}X\right) \bb{ for }  X\in \R^n,
\end{equation*}
where $U$ is the standard {\it bubble} in $\R^{n}$ given by
\begin{equation}\label{def:B0}
	U(X):=\left(\frac{1}{1+\bns|X|^{2-s}} \right)^{\frac{n-2}{2-s}} \bb{ for } X\in \R^n, \bb{ with }\bns^{-1}:=(n-s)(n-2).
\end{equation}
It follows from \cite{Chouchu}  that any function $V\in C^{0}_{loc}(\R^{n})\cap C^{2}_{\loc}(\R^{n}\setminus\{0\})$ satisfying:
\begin{equation}\label{eq:1}
\left\{\begin{array}{ll}
\Delta_{\xi}\,V=\dfrac{V^{\crits-1}}{|X|^s} & \bb{ in } \R^n \backslash \{0\}, \\
V > 0  &\hbox{ in } \R^n \backslash \{0\},
\end{array}\right.
\end{equation}
equals  $U_{\mu}$ for some $\mu>0$. Here, $\Delta_{\xi}$ is the usual Euclidean Laplacian.  It follows that the set of solutions $U_{\mu}$ is non-compact. 
\smallskip

We remark here that, any nonnegative $u\in C^{0,\theta}(M)\cap C^{2}(M \setminus \{\xo\})$, with $0<\theta<\min\{1,2-s\}$, satisfying $\Delta_{g}u + hu = \dfrac{u^{\crits-1}}{d_{g}(\xo,x)^{s}}$ in $M\setminus\{\xo\}$ indeed satisfies the maximum principle, that is,  either $u>0$ in $M$ or $u\equiv0$ in $M$. See for instance Jaber \cite{J1} (Step 6 in Page 45).
\medskip

\section{Proof of main results}\label{proofmaintheo}
Once an a priori bound is established, our theorems will follow from standard elliptic estimates.
\smallskip

\noindent\emph{Proof of Theorem \ref{main:theo1}}: The upper bound follows from  {\it standard elliptic estimates} if there exists $C>0$ such that $\Vert u \Vert_{L^{\infty}(M)} \le C$ for any solution of \eqref{eqn:HS1}, since $s<2$.

Suppose there exists a sequence of solutions $(u_{\alpha})_{\alpha\in\N}$ of \eqref{eqn:HS1} such that  $\Vert u_{\alpha}\Vert_{L^{\infty}(M)}\to+\infty$ as $\alpha\to+\infty$, that is, there exists a blowing up sequence.  Then by Proposition \ref{main:prop} we obtain $h_{\infty}(\xo)=\mathfrak{c}_{n,s}\sg(\xo)$ when $n\ge4$, where $\mathfrak{c}_{n,s}$ is the constant defined in \eqref{def:c}, and in dimension $3$ we obtain that $m_{h}(\xo)=0$, where $m_{h}(\xo)$ is the mass of the operator $\dg+h$. This contradicts the assumptions of Theorem \ref{main:theo1}, thereby establishing the upper bound.\par
\noindent
To obtain a uniform bound on $\left\|1/u\right\|_{C^{0}(M)}$ we again proceed by contradiction. Assume that there exists a sequence of solutions $(u_{\alpha})_{\alpha\in\N}$ of \eqref{eqn:HS1} and a sequence of points $(x_{\alpha})_{\alpha}\in M$ such that $u_{\alpha}(x_{\alpha})\to0$ as $\alpha\to+\infty$. Since the sequence of solutions $(u_{\alpha})_{\alpha\in\N}$ of \eqref{eqn:HS1} is such that  $\Vert u_{\alpha}\Vert_{L^{\infty}(M)}\le C$  for some constant $C>0$. We then obtain the existence of $u_{\infty}\in C^{0,\theta}(M)\cap C^{2}(M \setminus \{\xo\})$, where $0<\theta<\min\{1,2-s\}$,   such that  $u_{\alpha}\to u_{\infty}$ in $C^{0,\hat{\theta}}(M)\cap C^{2}_{\loc}(M \setminus \{\xo\})$ with $0\leq\hat{\theta}<\theta$. Furthermore $u_{\infty}\ge0$ in $M$  and satifisfies the equation:
\begin{align*}
\Delta_{g}u_{\infty} + h u_{\infty} = \dfrac{u_{\infty}^{\crits-1}}{d_{g}(\xo,x)^{s}}   & \hbox{ in }M\setminus\{\xo\}.
\end{align*}
Note that since for each $\alpha\ge1$ the quantity $d_{g}(x_{0},x)|\nabla \ua(x)|=\smallo(1)$ as $x\to x_{0}$, we also obtain in this case
\begin{align}
u_{\alpha}\to u_{\infty} \hbox{ in } H_1^2(M) \bb{ as } \alpha\to +\infty.
\end{align}
Since the operator $\Delta+h_{\infty}$ is coercive, so along the sequence of solutions $(u_{\alpha})_{\alpha\in\N}$ we obtain by applying the Hardy–Sobolev inequality that
\begin{align*}
\int_{M}\frac{u_{\alpha}^{\crits}}{d_{g}(\xo,x)^{s}}\,dv_{g}\ge C,
\end{align*}
for some $C>0$ independent of $\alpha$. The strong convergence $\lim\limits_{\alpha\to+\infty}u_{\alpha}\to u_{\infty}$ in $C^{0}$ then gives us $\displaystyle \int_{M}\frac{u_{\infty}^{\crits}}{d_{g}(\xo,x)^{s}}\,dv_{g}\ge C$ and this implies $u_{\infty}>0$ in $M$ by the maximum priciple as stated at the end of Section $2$. This prohibits the existence of a sequence such that $u_{\alpha}(x_{\alpha})\to0$ as $\alpha\to+\infty$ and this gives the uniform lower bound.
\hfill\qed
\medskip


\section{Controlling the Blow-up}\label{controlblowup}

Our aim in this section is to extract a necessary condition for blow-up to occur. This is presented in the following proposition. 
\begin{prop}\label{main:prop}
 Let $(M,g)$ be a closed and smooth Riemannian manifold of dimension $n\geq 3$, fix $\xo \in M$ and $s\in (0,2)$. Let $h_{\infty} \bb{ in } C^1(M)$ be such that the operator $\Delta_{g}+ h_{\infty}$ is {\it coercive}, and $h_{\infty}(x_{0})\le\mathfrak{c}_{n,s} \sg(\xo)$ when $n\ge4$ with $\mathfrak{c}_{n,s}$ given by \eqref{def:c}.
Consider a sequence of functions $(h_{\alpha})_{\alpha\in \N}\in C^1(M)$ such that $\lim\limits_{\alpha\to +\infty}h_{\alpha}= h_{\infty}$  in $C^1(M)$ and suppose  $(u_{\alpha})_{\alpha\in\N}\in C^{0,\theta}(M)\cap C^{2}(M \setminus \{\xo\})$, with $0<\theta<\min\{1,2-s\}$ satisfies:
\begin{equation}\label{eqn:HSua}
	\left\{\begin{array}{ll}
	\Delta_{g}u_{\alpha} + h_{\alpha}u_{\alpha} = \dfrac{u_{\alpha}^{\crits-1}}{d_{g}(\xo,x)^{s}}   & \hbox{ in }M\setminus\{\xo\},  \\
		\quad u_{\alpha} > 0  &\hbox{ in }M\setminus\{\xo\}.
	\end{array}\right.
\end{equation}
If  $(\ua)_{\alpha\in \N}$ {\it blows-up},  that is $\lim \limits_{\alpha \to+\infty} \Vert u_{\alpha}\Vert_{L^\infty(M)} =+\infty$, then 
 \begin{itemize}
 \item[$(i)$]  $h_{\infty}(\xo)=\mathfrak{c}_{n,s}\sg(\xo)$ when $n\ge4$.
 \item[$(ii)$] $m_{h_{\infty}}(\xo)=0$ when $n=3$.
\end{itemize}
Here $m_{h_{\infty}}(\xo)$ is the mass of the operator $\dg+h_{\infty}$ defined in \eqref{green0}, and $\mathfrak{c}_{n,s}$ is given by \eqref{def:c}. 
\end{prop}
\noindent
Note that coercivity is a necessary condition for the existence of positive solutions of \eqref{eqn:HSua} (see Appendix B of \cite{DHR}).
\smallskip

\noindent
We prove this proposition in multiple steps, starting with the following lemma, which implies that only $\xo$ can be the point of blow-up. 
\begin{lemma}\label{P1}
 Let $(M,g)$ be a closed and smooth Riemannian manifold of dimension $n\geq 3$, fix $\xo \in M$ and $s\in (0,2)$.  Let $(h_{\alpha})_{\alpha\in \N}$ be as in the statement of Proposition \ref{main:prop}. There exists  $C>0$ such that for any sequence $(u_{\alpha})_{\alpha\in\N}\in C^{0,\theta}(M)\cap C^{2}(M \setminus \{\xo\})$ satisfying  equation \eqref{eqn:HSua}, with $0<\theta<\min\{1,2-s\}$, we have 
\begin{align}\label{es1}
d_g(\xo,x)^{\frac{n-2}{2}}u_{\alpha}(x)\leq C~\hbox{ for all } x \in M.
\end{align}
\end{lemma}
\begin{proof}
We proceed by contradiction: assume that there exists a sequence $(y_{\alpha})_{\alpha\in \N}\in M$ such that 
\begin{equation}\label{cond:initial}
\sup_{M}\(d_g(\xo,x)^{\frac{n-2}{2}}u_{\alpha}(x)\)=:d_g(\xo,y_{\alpha})^{\frac{n-2}{2}}u_{\alpha}(y_{\alpha})\to+\infty,
\end{equation}
as $\alpha\to +\infty$. Since $M$ is compact, this implies
\begin{equation}\label{uatoinf}
\lim_{\alpha\to +\infty} \ua(y_\alpha)=+\infty.
\end{equation}
Now, let $\nna:=u_{\alpha}(y_{\alpha})^{2/(2-n)}$ for all $\alpha\geq 1$. Then from \eqref{cond:initial} and \eqref{uatoinf} it follows that
\begin{equation*}
\lim \limits_{\alpha\to+\infty} \frac{d_{g}(\xo,y_{\alpha})}{\nna}=+\infty \bb{ and }\lim \limits_{\alpha\to+\infty}\nna=0.
\end{equation*}
Again, letting $\ell_{\alpha}:= d_{g}(\xo,y_{\alpha})^{s/2}\nna^{\frac{2-s}{2}}$ for all $\alpha >0$, since $M$ is compact, we have
\begin{equation}\label{convlalpha}
\lim \limits_{\alpha\to+\infty} \frac{d_{g}(\xo,y_{\alpha})}{\ell_{\alpha}}=+\infty \bb{ and }\lim \limits_{\alpha\to+\infty}\ell_{\alpha}=0.
\end{equation}
We rescale and define for all $\alpha\in \N$
\begin{equation*}
	\bar{v}_{\alpha}(X):= \nna^{\frac{n-2}{2}}u_{\alpha}\left(\exp_{y_{\alpha}}(\ell_{\alpha} X)\right) \bb{ for all  }X\in  B(0,\ig/2\ell_{\alpha}).
\end{equation*}
Here $\exp_{y_{\alpha}}$ is the exponential chart at $y_\alpha$ with respect to the metric $\bar{g}_{\alpha}:=\exp_{y_{\alpha}}^{*} g(\ell_{\alpha}\cdot)$  for all $\alpha\geq 1$. 
\smallskip

\noindent
 From the definition we have that  $\bar{v}_{\alpha}(0)=1$ for all $\alpha\in \N$ and using \eqref{cond:initial} we obtain
\begin{align}\label{eq,bd}
\bar{v}_{\alpha}(X)\leq \left(\frac{d_{g}(\xo,y_{\alpha})}{d_{g}(\xo,\exp_{y_{\alpha}}(\ell_{\alpha}X))}\right)^{\frac{n-2}{2}}~ \hbox{~in }B(0,\ig/2\ell_{\alpha}).
\end{align}
Note for $R>0$ we have for $\alpha$ large
\begin{equation*}
|d_{g}(\xo,\exp_{y_{\alpha}}(\ell_{\alpha}X))-d_{g}(\xo,y_{\alpha})|\le\ell_{\alpha} R~ \bb{ for all } X\in B(0,R),
\end{equation*}
and with the help of \eqref{convlalpha} this gives us
\begin{equation}\label{small1}
	\lim_{\alpha\to + \infty}\left(\frac{d_{g}(\xo,y_{\alpha})}{d_{g}(\xo,\exp_{y_{\alpha}}(\ell_{\alpha}X))}\right)^{\frac{n-2}{2}}= 1.
\end{equation} 
Then in \eqref{eq,bd} we obtain that as  $\alpha \to +\infty$ 
\begin{align}\label{eq30}
	\bar{v}_{\alpha}(X)\leq  \left(\frac{d_{g}(\xo,y_{\alpha})}{d_{g}(\xo,\exp_{y_{\alpha}}(\ell_{\alpha}X))}\right)^{\frac{n-2}{2}}\leq 1+\smallo(1)~ \hbox{~in }B(0,R),
\end{align}
where $\smallo(1)\to 0$ as $\alpha\to +\infty$.
\smallskip

\noindent
From equation \eqref{eqn:HSua}, it then follows that for all $R>0$ the function $\bar{v}_{\alpha}$ satisfies the equation: 
\begin{align*}
\Delta_{\bar{g}_{\alpha}}\bar{v}_{\alpha}+\ell_{\alpha}^{2}h_{\alpha}(\exp_{y_{\alpha}}(\ell_{\alpha}X))\bar{v}_{\alpha}=\left(\frac{d_{g}(\xo, y_{\alpha})}{d_{g}(\xo, \exp_{y_{\alpha}}(\ell_{\alpha}x))}\right)^{s}\bar{v}_{\alpha}^{\,2^{*}(s)-1}\\
~\hbox{ in } B(0,R).
\end{align*}
By standard elliptic theory along with \eqref{convlalpha}, \eqref{small1} and \eqref{eq30}, we can get that $\bar{v}_{\alpha}\to \bar{v}_{R,\infty}$ in $C^{0}(B(0,R/2))\cap C^{2}_{\loc}(B(0,R/2)\setminus\{0\})$ as $\alpha\to+\infty$, where
\begin{equation*}
\left\{\begin{array}{ll}
\quad\Delta_{\xi}\,\bar{v}_{R,\infty}=\bar{v}_{R,\infty}^{\,2^{*}(s)-1}& \hbox{ in }B(0,R/2),  \\
0\le\bar{v}_{R,\infty}\le1=\bar{v}_{R,\infty}(0) &\hbox{ in }B(0,R/2).
\end{array}\right.
\end{equation*}
Letting $R\to +\infty$, we obtain $\bar{v}_{\infty}$ in $C^{2}_{\loc}(\R^{n}))$ satisfying
\begin{equation*}
\left\{\begin{array}{ll}
\quad\Delta_{\xi}\,\bar{v}_{\infty}=\bar{v}_{\infty}^{\,2^{*}(s)-1}& \hbox{ in }\R^{n},  \\
0\le\bar{v}_{\infty}\le1=\bar{v}_{\infty}(0) &\hbox{ in }\R^{n}.
\end{array}\right.
\end{equation*}
Since $s>0$ this is impossible by the Liouville type result of Gidas and Spruck \cite{GS}, and we get a contradiction. This proves \eqref{es1}.
\end{proof}
\begin{remark}
The previous proposition captures all blow-up profiles and says that the only blow-up point is the point $\xo$. By standard elliptic theory then there exists $u_{\infty}\in C^{0}(M)\cap C^{2}(M \setminus \{\xo\})$ such that 
\begin{equation}
	\lim_{\alpha\to +\infty} \ua = u_{\infty} \bb{ in }C_{loc}^{2}(M\setminus\{\xo\}).
\end{equation} 
\end{remark}

\noindent
Next, we show that the blowing-up sequence behaves like the rescalings of the {\emph{standard bubble}} around the blow-up point $\xo$.
\begin{lemma}\label{P2}
Let $(M,g)$ be a closed and smooth Riemannian manifold of dimension $n\geq 3$, fix $\xo \in M$ and $s\in (0,2)$.  Let $(h_{\alpha})_{\alpha\in \N}$ be as in the statement of Proposition \ref{main:prop} and consider $(u_{\alpha})_{\alpha\in \N}$ satisfying  equation \eqref{eqn:HSua} such that $\lim \limits_{\alpha \to+\infty} \Vert u_{\alpha}\Vert_{L^\infty(M)} =+\infty$. Then $\mu_{\alpha}:=u_{\alpha}(\xo)^{2/(2-n)}\to0$ as $\alpha\to+\infty$, and rescalings
\begin{equation}\label{defuta}
	\ut(X):=\mu_{\alpha}^{\frac{n-2}{2}}u_{\alpha}\(\exp_{\xo}(\ma X)\) \bb{ for all } X\in B(0,\ig/2\ma )
\end{equation}
satisfies
\begin{align}\label{strongconvua}
\lim_{\alpha\to +\infty}\ut= U~\hbox{ in } C_{loc}^{0}(\R^n)\cap C_{loc}^{2}(\R^{n}\setminus\{0\}).
\end{align}
Here $U$ is defined as in \eqref{def:B0} and satisfies the equation \eqref{eq:1}.
\end{lemma}
\begin{proof}
Indeed, we let  $(x_{\alpha})_{\alpha\in \N}\in M$ be a sequence such that 
$$u_{\alpha}(x_{\alpha}):=\max\limits_{x\in M}u_{\alpha}(x)\to+\infty~\hbox{ as }\alpha\to+\infty,$$
 and we take $\check{\mu}_{\alpha}:=u_{\alpha}(x_{\alpha})^{2/(2-n)}$. From Lemma \ref{P1} it follows that $d_{g}(\xo,x_{\alpha})=\bigO(\check{\mu}_{\alpha})$ as $\alpha\to+\infty$ and hence $\lim \limits_{\alpha\to+\infty}x_{\alpha}=\xo$. We rescale and define for $\alpha\ge1$ 
\begin{equation*}
\widetilde{u}_{\alpha}(X):=\check{\mu}_{\alpha}^{\frac{n-2}{2}}u_{\alpha}\big(\exp_{\xo}(\check{\mu}_{\alpha} X)\big) ~ \hbox{ for } X \in B(0,\iota_{g}/2\check{\mu}_{\alpha}).
\end{equation*}
where $\exp_{\xo}$ is the exponential chart at $\xo$ with respect to the metric $\widetilde{g}_{\alpha}:=\exp_{\xo}^{*} g(\check{\mu}_{\alpha}\cdot)$. It then follows that 
\begin{align*}
0\le\widetilde{u}_{\alpha}(x)\leq 1 \bb{ in }B(0,\iota_{g}/2\check{\mu}_{\alpha})~\hbox{ and } \widetilde{u}_{\alpha}(\widetilde{X}_{\alpha}) =1,
\end{align*}
where $\widetilde{X}_{\alpha}=\check{\mu}_{\alpha}^{-1}\exp_{\xo}^{-1}x_{\alpha}$. Again by lemma \ref{P1} we obtain $\widetilde{X}_{\infty}\in \R^{n}$ such that $\lim\limits_{\alpha\to+\infty}\widetilde{X}_{\alpha}=\widetilde{X}_{\infty}$. Next, from equation \eqref{eqn:HSua},  it follows that the rescaled functions $\widetilde{u}_{\alpha}$ satisfies for $\alpha$ large: 
\begin{align*}
\Delta_{\widetilde{g}_{\alpha}}\widetilde{u}_{\alpha}+\check{\mu}_{\alpha}^{2}h_{\alpha}(\exp_{\xo}(\check{\mu}_{\alpha}X))\widetilde{u}_{\alpha}=\frac{\widetilde{u}_{\alpha}^{\,2^{*}(s)-1}}{|X|^{s}} ~\hbox{ in } B(0,\iota_{g}/2\check{\mu}_{\alpha})\setminus\{0\}.
\end{align*}
By standard elliptic theory along with a diagonal argument, we get that $\widetilde{u}_{\alpha}\to\widetilde{u}_{\infty}$ in $ C_{loc}^{0}(\R^n)\cap C_{loc}^{2}(\R^{n}\setminus\{0\})$ as $\alpha\to+\infty$, where
\begin{align*}
\Delta_{\xi}\widetilde{u}_{\infty}=\frac{\widetilde{u}_{\infty}^{\,2^{*}(s)-1}}{|X|^{s}}~\hbox{ in } \R^{n}\setminus\{0\}, \hbox{ with } 0\le\widetilde{u}_{\infty}\leq 1 \bb{ in }\R^{n}.
\end{align*}
Note that $\widetilde{u}_{\infty}(\widetilde{X}_{\infty})=1$. From the classification result of Chou-Chu \cite{Chouchu} we obtain that
\begin{align*}
\widetilde{u}_{\infty}(X)=\frac{\mu^{\frac{n-2}{2}}}{(\mu^{2-s}+\mathfrak{b}_{n,s}|X|^{2-s})^{\frac{n-2}{2-s}}}~\hbox{ for some } \mu>0, 
\end{align*}
where $\mathfrak{b}_{n,s}$ is defined in \eqref{def:B0}. It follows $\mu^{\frac{2-n}{2}}=\widetilde{u}_{\infty}(0)\leq1$ and $1=\widetilde{u}_{\infty}(\widetilde{X}_{\infty})$ implies $\mu^{\frac{2-s}{2}}(\mu^{\frac{2-s}{2}}-1)+\mathfrak{b}_{n,s}|\widetilde{X}_{\infty}|=0$. Therefore
\begin{align*}
\widetilde{u}_{\infty}(X)=\(1+\mathfrak{b}_{n,s}|X|^{2-s}\)^{\frac{2-n}{2-s}}~\hbox{ and }  \widetilde{X}_{\infty}=0.
\end{align*}
This implies that  $d_{g}(\xo,x_{\alpha})=\smallo(\check{\mu}_{\alpha})$ as $\alpha\to +\infty$ and
\begin{align*}
1=\widetilde{u}_{\infty}(0)=\lim\limits_{\alpha \to+\infty}\widetilde{u}_{\alpha}(0)=\lim\limits_{\alpha \to+\infty}\check{\mu}_{\alpha}^{\frac{n-2}{2}}u_{\alpha}(\xo).
\end{align*}
Thus we can replace $\check{\mu}_{\alpha}$ with $\mu_{\alpha}$  in the above analysis, completing the proof of Lemma \ref{P2}.
\end{proof}
\medskip

We now define the {\it rescaled bubble} centered at $\xo$ with height $\sim \mu^{(2-n)/2}_{\alpha}$ as:
\begin{equation}\label{def:bubble}
B_{\alpha}(x):=\frac{\mu_{\alpha}^{\frac{n-2}{2}}}{(\mu_{\alpha}^{2-s}+\mathfrak{b}_{n,s}d_{g}(\xo,x)^{2-s})^{\frac{n-2}{2-s}}}
\end{equation}
where  $\mu_{\alpha}:=u_{\alpha}(\xo)^{2/(2-n)}$ and $\mathfrak{b}_{n,s}^{-1}=(n-s)(n-2)$. \par 
\smallskip\noindent It then follows from Lemma \ref{P2} that for any fixed $R>0$ we have $$u_{\alpha}=\(1+\smallo(1)\)B_{\alpha}\,\bb{ in } C^{0}(B_g(\xo, R\ma)) \bb{ as } \alpha \to +\infty.$$ 
Note, for each fixed $\alpha\ge1$ we have by the H\"older continuity of the $u_{\alpha}$, that  $d_{g}(x_{0},x)|\nabla \ua(x)|=\smallo(1)$ as $x\to x_{0}$. Then for any fixed $R>0$ we also obtain
\begin{align}\label{eq5}
\tua \to U\bb{ in } H_1^2(B(0,R)) \bb{ as } \alpha\to +\infty.
\end{align}
Here $0<\theta<\min\{1,2-s\}$.
\smallskip

\noindent
In fact,  we will show that the blowing up sequence $(u_{\alpha})_{\alpha\in \N}$ is well approximated by the standard bubble in a non-collapsing ball around $\xo$.  But first, assuming control by a bubble, we will use a  Pohozaev identity to get a balancing condition in case blow-up occurs.  For convinience we denote 
\begin{equation}\label{defpo}
	\left\{\begin{array}{ll}
	\varphi_{\alpha}(\xo):=h_{\alpha}(\xo)-\mathfrak{c}_{n,s}\sg(\xo),  \\\\
		\varphi_{\infty}(\xo):=h_{\infty}(\xo)-\mathfrak{c}_{n,s}\sg(\xo),
	\end{array}\right.
\end{equation}
where $\mathfrak{c}_{n,s}$ is defined in \eqref{def:c}.
\begin{lemma}\label{lem:poho}
Let $(M,g)$ be a closed and smooth Riemannian manifold of dimension $n\geq 3$, fix $\xo \in M$ and $s\in (0,2)$.  Let $(h_{\alpha})_{\alpha\in \N}$ be as in the statement of Proposition \ref{main:prop} and consider $(u_{\alpha})_{\alpha\in \N}$ satisfying  equation \eqref{eqn:HSua} such that $\lim \limits_{\alpha \to+\infty} \Vert u_{\alpha}\Vert_{L^\infty(M)} =+\infty$. Consider a sequence of positive radius $(\kappa_{\alpha})_{\alpha\in \N}\in (0,\iota_{g}/4)$ such that $\lim\limits_{\alpha\to +\infty} \dfrac{\kappa_{\alpha}}{\mu_{\alpha}}=+\infty$. For $R>0$, we assume that 
\begin{align}\label{acontrol1}
u_{\alpha}(x)&\le C B_{\alpha}(x)~\hbox{ for all }x\in B_g(\xo,4\kappa_{\alpha}),\notag\\
\kappa_{\alpha}|\nabla u_{\alpha}(x)|&\le C_{R} B_{\alpha}(x)~\hbox{ for all }x\in B_g(\xo,4\kappa_{\alpha})\setminus B_g(\xo,\kappa_{\alpha}/R),
\end{align}
for some constants $C>0$, $C_{R}>0$, and for any $\varepsilon \in (0,1)$
\begin{align}\label{acontrol2}
\left|\frac{u_{\alpha}(x)}{B_{\alpha}(x)}-1\right|\le\varepsilon \bb{ for all } x\in B_g(x_0,\kappa_\alpha).
\end{align}
Here $B_{\alpha}$ is the rescaled bubble defined above in \eqref{def:bubble}.\par  
\smallskip\noindent Then we have as $\alpha\to+\infty$
\begin{align}\label{blow rates}
\kappa_{\alpha}^2\ln\left( \frac{\kappa_{\alpha}}{\ma}\right)\varphi_{\alpha}(\xo)\omega_3\bn^{-\frac{4}{2-s}}&=\bigO\left(1\right) &\hbox{ if } n=4,\notag\\
 \varphi_{\alpha}(\xo)\int_{\R^n} U^{2}\, dX&=\bigO\left(\frac{\mu_{\alpha}^{n-4}}{\kappa_{\alpha}^{n-2}}\right)&\hbox{ if } n\geq 5.
\end{align}
Here, $\omega_{3}$ denotes the area of the unit sphere in $\R^3$, and $\bn$ is defined in \eqref{def:B0}. The function $U$ is given by \eqref{def:B0} and $\varphi_{\alpha}$ is defined in \eqref{defpo}.
\end{lemma}
\begin{proof}
We follow the approach of Cheikh-Ali \cite{HCA4} and use a Pohozaev identity to obtain the above blow-up rates \eqref{blow rates}. Recall the celebrated Riemannian version of the Pohozaev identity. Let $\Omega \subset \R^n$  be smooth bounded domain, and let  $u\in C^2(\overline{\Omega})$ with $u>0$. Then for all $Z\in \R^n$ and $l\in \{1,..., n\}$, the Pohozaev identity (see Hebey \cite{Hbooks}, Ghoussoub-Robert \cite{GRHSinterior}) can be written as:
\begin{align}\label{PoIden}
	&\int_{\Omega} \left( \left( X-Z\right)^{l} \partial_{l}u+\frac{n-2}{2}u \right)\left(\Delta_{\xi} u-\frac{u^{ \crits-1}}{|X|^s}\right) \, dX\nonumber\\
	&=\int_{\partial \Omega} \left( X-Z, \nu\right) \left( \frac{|\nabla u |^2}{2} -\frac{1}{\crits}\frac{u^{ \crits}}{|X|^s}\right)\, d\sigma(X)\\
	&- \int_{\partial \Omega}\left(\left( X-Z\right) ^l \partial_l u+\frac{n-2}{2}u \right) \partial_{\nu}u \, d\sigma(X),\nonumber
\end{align}
where $\nu(X)$ is the outer unit normal to the boundary of $\Omega$. Let 
\begin{equation}\label{def:uha}
	\widehat{u}_{\alpha}(X):= u_{\alpha}(\exp_{\xo}(X))\bb{ for all } X\in B(0,4\kappa_\alpha)\subset \R^n. 
\end{equation}
Equation \eqref{eqn:HSua} becomes:
\begin{equation}\label{eq:uha0}
\Delta_{\widehat{g}}\,\uha + \widehat{h}_{\alpha}   \uha=\frac{ \uha^{\, \crits-1}}{|X|^s} \bb{ in }  B(0,4\kappa_\alpha)\setminus\{0\},
\end{equation}
where $\ha(X)= h_{\alpha}(\exp_{\xo}( X))$ and $\widehat{g}(X):=( \exp^*_{\xo} g) (X)$ is the pull back of $g$ via the exponential map. 
Applying \eqref{PoIden} we  obtain the following Pohozaev identity
\begin{align}\label{PoIden2}
&\int_{\partial B(0,\kappa_\alpha)} \left( X, \nu\right) \left( \frac{|\nabla \uha |^2}{2} -\frac{1}{\crits}\frac{\uha^{\, \crits}}{|X|^s}\right)\,d\sigma(X)\notag\\
&~- \int_{\partial B(0,\kappa_\alpha)}\left(X^l \partial_l \uha+\frac{n-2}{2}\uha \right) \partial_{\nu}\uha \, d\sigma(X)\notag\\
=&-\int_{B(0,\kappa_\alpha)} \left( X^l \partial_l \uha+\frac{n-2}{2}\uha\right)\ha\uha\, dX\notag\\
&~-\int_{B(0,\kappa_\alpha)} \left( X^l \partial_l \uha+\frac{n-2}{2}\uha \right)\left(\Delta_{\widehat{g}}\,\uha -\Delta_{\xi}\uha \right)\, dX.
\end{align}
Note, the integrals above makes sense, since one has
\begin{equation*}
	\lim\limits_{|X|\to0}|X||\nabla \uha(X)|=0 \bb{ and }\lim\limits_{|X|\to0}|X|^{2}|\nabla^{2} \uha(X)|=0,
\end{equation*}
for all $\alpha\ge1$, and by using the Cartan's expansion of the metric one obtains for some $C>0$, that
$$\left|\Delta_{\widehat{g}}\,\uha -\Delta_{\xi}\uha \right|\leq C\(|X||\nabla \uha(X)|+|X|^{2}|\nabla^{2} \uha(X)|\).$$ 
\noindent
Using the assumed bound \eqref{acontrol1}, we get as $\alpha\to +\infty$ that 
\begin{align}\label{ralpha}
&\int_{\partial B(0,\kappa_\alpha)} \left( X, \nu\right) \left( \frac{|\nabla \uha |^2}{2} -\frac{1}{\crits}\frac{\uha^{\, \crits}}{|X|^s}\right)\,d\sigma(X)\notag\\
&~- \int_{\partial B(0,\kappa_\alpha)}\left(X^l \partial_l \uha+\frac{n-2}{2}\uha \right) \partial_{\nu}\uha \, d\sigma(X)=\bigO\(\frac{\mu_{\alpha}^{n-2}}{\kappa_{\alpha}^{n-2}}\).
\end{align}
By \eqref{acontrol2} and since $h_\alpha\to h_{\infty}$ in $C^1(M)$ as $\alpha\to +\infty$, following the calculations in Cheikh-Ali \cite{HCA4} we obtain as $\alpha\to +\infty$
\begin{align}\label{calpha}
&\int_{B(0,\kappa_\alpha)} \left( X^l \partial_l \uha+\frac{n-2}{2}\uha\right)\ha\uha\, dX\notag\\
&~=\left\{\begin{array}{ll}  
\bigO\left(\kappa_{\alpha}\ma\right)  &\bb{ for } n=3\\\\
-\ma^2\ln\left(\dfrac{\kappa_{\alpha}}{\mu_{\alpha}}\right)\left(  \omega_3 \bns^{-\frac{4}{2-s}}h_{\infty}(\xo)+\smallo(1)\right) &\hbox{ for }n=4 \\\\
-\ma^2 \( h_{\infty}(\xo)\displaystyle\int_{\R^n} U^2 \, dX+\smallo(1)\)  &\hbox{ for } n\geq 5.
\end{array}\right.
\end{align}
From \eqref{strongconvua}, for any $R>0$ we have  $\tua \to U\bb{ in } H_1^2(B(0,R)) \bb{ as } \alpha\to +\infty$. Using \eqref{acontrol1} and following Step~4.11 in Cheikh-Ali~\cite{HCA4}, we get as $\alpha\to+\infty$
\begin{align}\label{dalpha}
&\int_{B(0,\kappa_\alpha)} \left( X^l \partial_l \uha+\frac{n-2}{2}\uha \right)  \left(\Delta_{\widehat{g}}\,\uha -\Delta_{\xi}\uha  \right) \, dX\notag\\
&=\left\{\begin{array}{ll}  
\bigO\left(\kappa_{\alpha}\ma\right)  &\bb{ for } n=3\\\\
\ma^2\ln\left(\dfrac{\kappa_{\alpha}}{\mu_{\alpha}}\right)\left(\dfrac{\omega_3}{6}\bn^{-\frac{4}{2-s}}\sg(\xo)+\smallo(1)\right)&\hbox{ for }n=4\\\\
\ma^2\(\mathfrak{c}_{n,s}\sg(\xo) \displaystyle\int_{\R^n} U^2 \, dX+\smallo(1)\)  &\hbox{ for } n\geq 5.
\end{array}\right.
\end{align}
Note, one needs to make use the following estimate (Lemma~4.3 in Cheikh-Ali~\cite{HCA4}) in dimension $n=4$. Consider $\ut$ defined in \eqref{defuta}.  For $1\le i_{1}, i_{2},j_{1},j_{2}\le n=4$, we have
\begin{equation}\label{key4}
 \lim\limits_{\alpha\to +\infty}\frac{\displaystyle\int_{B(0,r_{\alpha}/\mu_{\alpha})} X^{i_{1}}X^{i_{2}} \partial_{j_{1}} \ut \partial_{j_{2}} \ut \,dX}{\ln\(\frac{\kappa_{\alpha}}{\mu_{\alpha}}\)}= 4\bn^{-\frac{4}{2-s}}\int_{\mathbb{S}^{3}}\sigma^{i_{1}} \sigma^{i_{2}} \sigma^{j_{1}}\sigma^{j_{2}} \, d\sigma,
 \end{equation}
where $\bn$ is  defined in \eqref{def:B0}. Plugging \eqref{ralpha}, \eqref{calpha} and \eqref{dalpha} in the Pohozaev identity \eqref{PoIden2} gives us the claimed estimates in equation \eqref{blow rates}.  The proof of the lemma is complete.
\end{proof}
Following is the core technical result of this paper. The ideas and the techniques are motivated from Druet and Premoselli \cite{DruetPremoselli},
Premoselli \cite{Bruno} and Premoselli-V\'etois \cite{PV}. Also see Druet and Laurain \cite{DruetLaurain} for the case $s=0$, and  Cheikh-Ali \cite{HCA4} for the case when the blowing up sequence is close to one bubble in the energy.
\begin{prop}\label{P3}
Let $(M,g)$ be a closed and smooth Riemannian manifold of dimension $n\geq 3$, fix $\xo \in M$ and $s\in (0,2)$.  Let $(h_{\alpha})_{\alpha\in \N}$ be as in the statement of Proposition \ref{main:prop} and consider $(u_{\alpha})_{\alpha\in \N}$ satisfying  equation \eqref{eqn:HSua} such that $\lim \limits_{\alpha \to+\infty} \Vert u_{\alpha}\Vert_{L^\infty(M)} =+\infty$. Then for any $\varepsilon\in(0,1)$ there exists $\rho_{\eps}>0$ small such that we have
\begin{align}\label{es:main}
\left\|\frac{u_{\alpha}}{B_{\alpha}}-1\right\|_{L^{\infty}(B_g(\xo,\rho_{\eps}))}\le\varepsilon.
\end{align}
Here $B_{\alpha}$ is the rescaled bubble defined in \eqref{def:bubble}. 
\end{prop}

\begin{proof}
\noindent
We define the \emph{radius of influence} $r_\alpha>0$ of the standard bubble with $\varepsilon\in (0,1)$ error as:
\begin{align}\label{def:r}
r_{\alpha}:=&\sup\big\{ \mu_{\alpha}\leq r\leq \iota_{g}/6: |u_{\alpha}(x)-B_{\alpha}(x)|\leq\varepsilon B_{\alpha}(x)~\forall\,x\in B_g(\xo,r)\big\}.
\end{align}
From \eqref{strongconvua} it follows that
\begin{equation}\label{limrama}
	\lim_{\alpha\to +\infty} \dfrac{r_{\alpha}}{\mu_{\alpha}}=+\infty.
\end{equation}
We start with establishing the following control. 
\begin{step}\label{s1} There exists $C>0$ such that for all $\alpha$ large one has 
\begin{align}\label{main:ineq1}
u_{\alpha}(x)&\le C B_{\alpha}(x)~\hbox{ for all }x\in B_g(\xo,4r_{\alpha}),
\end{align}
and for all $R>0$ there exists $C_{R}>0$ such that for all $\alpha\ge1$ large
\begin{align}\label{main:ineq2}
r_{\alpha}|\nabla u_{\alpha}(x)|&\le C_{R} B_{\alpha}(x)~\hbox{ for all }x\in B_g(\xo,4r_{\alpha})\setminus B_g(\xo,r_{\alpha}/R).
\end{align}
\end{step}

\noindent
{\emph{Proof of Step \ref{s1}}:} 
Note from the definition of $r_{\alpha}$ it follows that 
\begin{equation}\label{eq4}
	u_{\alpha}(x)\le C_1\,B_{\alpha}(x) \bb{ for all } x\in B_g(0,r_\alpha),
\end{equation}
for some $C_1>0$ independent of $\alpha$. Fix $R>1$. Applying the  Harnack's inequality on $B_g(\xo,4r_{\alpha})\setminus B_g(\xo,r_{\alpha}/R)$ as in lemma 6.2 of Hebey \cite{Hbooks} (with minor modifications), we obtain for all $x\in B_g(\xo,4r_{\alpha})\setminus B_g(\xo,r_{\alpha}/R) $
\begin{align*}
&u_{\alpha}(x)\leq\max\limits_{B(0,4)\setminus B(0,1/R)} u_{\alpha}(\exp_{\xo}(r_{\alpha}\,\cdot))\\
&~\le C_2\,r_{\alpha}^{\frac{2-n}{2}}\,\min\limits_{B(0,4)\setminus B(0,1/R)} r_{\alpha}^{\frac{n-2}{2}}u_{\alpha}(\exp_{\xo}(r_{\alpha}\,\cdot))\notag\\
&~\le C_3\,r_{\alpha}^{\frac{2-n}{2}}\, B_{\alpha}(x),
\end{align*}
for some positive constants  $C_2$ and $C_3$.  Together with \eqref{eq4} this gives \eqref{main:ineq1}. The gradient bound \eqref{main:ineq2} follows from standard elliptic estimates.
\hfill\qed
\smallskip

We now prove convergence to the fundamental solution on $\R^{n}$.
\begin{step}\label{s3} Assume $r_{\alpha}\to0$ as $\alpha\to+\infty$.  We rescale and define
\begin{align*}
w_{\alpha}(X):=\frac{r_{\alpha}^{n-2}}{\mu_{\alpha}^{\frac{n-2}{2}}}u_{\alpha}(\exp_{x_{0}}(r_{\alpha}X))~\hbox{ for } X\in B(0,2).
\end{align*}
Then $w_{\alpha}(X)\to[(n-2)(n-s)]^{\frac{n-2}{2-s}}|X|^{2-n}$ in $C^{2}_{\loc}(B(0,2)\setminus\{0\})$ as $\alpha\to+\infty$.
\end{step}
\noindent
{\emph{Proof of Step \ref{s3}} :} Let $g_{\alpha}:=\exp_{x_{0}}^{*} g(r_{\alpha}\cdot)$ for $\alpha\ge1$. It follows, using \eqref{eqn:HSua}, that $w_{\alpha}$ satisfies: 
\begin{equation}\label{eq:uha1}
\Delta_{g_{\alpha}}w_{\alpha}+r_{\alpha}^{2}h_{\alpha}(\exp_{x_{0}}(r_{\alpha}X))\,w_{\alpha}=\(\frac{\mu_{\alpha}}{r_{\alpha}}\)^{2-s}\,\frac{w_{\alpha}^{\,2^{*}(s)-1}}{|X|^{s}} ~\hbox{ in } B(0,2)\setminus\{0\},
\end{equation}
and from Step \ref{s1} we have for some constant $C>0$
\begin{equation}\label{cont:uha}
	w_{\alpha}(X)\leq C\,|X|^{2-n} \bb{ in }B(0,2)\setminus \{0\}.
\end{equation}
Standard elliptic estimates implies then  $w_{\alpha}\to w_{\infty}$ in $C^{1}_{\loc}(B(0,2)\setminus\{0\})$ as $\alpha\to+\infty$, where
\begin{equation*}
		\Delta_{\xi}w_{\infty}=0 \bb{ and }	w_{\infty}(X)\leq C\,|X|^{2-n} \bb{ in } B(0,2)\setminus\{0\},
\end{equation*}
for some constant $C>0$.
By the B\^{o}cher theorem \cite{B,ABR} on the singularities of harmonic functions we can write as $ \alpha\to +\infty$,
\begin{align}\label{sing+har}
w_{\alpha}\to w_{\infty}=\mathscr{A} |X|^{2-n}+\mathscr{H}(X)~\hbox{ in } C^{1}_{\loc}(B(0,2)\setminus\{0\}),
\end{align}
where $\mathscr{A}$ is a constant and $\Delta_\xi \mathscr{H}=0$ in $B(0,2)$. 
\begin{claim}
We claim that in \eqref{sing+har}
\begin{equation}\label{constantLambda}
\mathscr{A}=\mathfrak{b}_{n,s}^{-\frac{n-2}{2-s}}=\[(n-2)(n-s)\]^{\frac{n-2}{2-s}}.
\end{equation}
\end{claim}
\noindent{\it Proof of the claim}: For any $\delta>0$ small, integrating equation \eqref{eq:uha1} in $B(0,1)\setminus B(0,r_{\alpha}^{2})$ we obtain
\begin{eqnarray}\label{po1}
&&\int_{B(0,1)\setminus B(0,\delta\mu_{\alpha}/r_{\alpha})}\Delta_{g_{\alpha}}w_{\alpha}\, dv_{g_{\alpha}}= \left( \frac{\ma}{\ra}\right)^{2-s}\int_{B(0,1)\setminus B(0,\delta\mu_{\alpha}/r_{\alpha})}\frac{w_{\alpha}^{\, \crits-1}}{|X|^s}\, dv_{g_{\alpha}}\notag\\
&&\quad-\,\ra^2\int_{B(0,1)\setminus B(0,\delta\mu_{\alpha}/r_{\alpha})}h_{\alpha}(\exp_{x_{0}}(r_{\alpha}X))\,w_{\alpha} \, dv_{g_{\alpha}}.
\end{eqnarray}
Using \eqref{cont:uha} 
\begin{equation*}
\int_{B(0,1)\setminus B(0,\delta{\mu_{\alpha}}/{r_{\alpha}})} h_{\alpha}(\exp_{x_{0}}(r_{\alpha}X))w_{\alpha}\, dv_{g_{\alpha}}=\bigO\left( \int_{B(0,1)}  |X|^{2-n} \, dX\right)=\bigO(1).
\end{equation*}
And by a change of variable, we obtain using \eqref{strongconvua} 
\begin{align*}
&\lim_{\alpha\to +\infty} \left( \frac{\ma}{\ra}\right)^{2-s}\int_{B(0,1)\setminus B(0,\delta\mu_{\alpha}/r_{\alpha})} \frac{w_{\alpha}^{\, \crits-1}}{|X|^s}\, dv_{g_{\alpha}}\notag\\
&~=\lim_{\alpha\to +\infty}\int_{B\left( 0,{\ra}/{\ma}\right)\setminus B(0,\delta)} \frac{\tua^{\,\crits-1}}{|X|^s}\, dv_{g_{\alpha}}=\int_{\R^{n}\setminus B(0,\delta)}\frac{U^{\crits-1}}{|X|^s}\, dX.
\end{align*}
Here $U$ is defined in \eqref{def:B0}. Next, we write 
\begin{align*}
\int_{B(0,1)\setminus B(0,\delta\mu_{\alpha}/r_{\alpha})}\Delta_{g_{\alpha}}w_{\alpha}\, dv_{g_{\alpha}}=-\int_{\partial B(0,1)} \pr_{\nu}w_{\alpha}\,
d\sigma_{g_{\alpha}}+\int_{\partial B(0,\delta\mu_{\alpha}/r_{\alpha})} \pr_{\nu}w_{\alpha}\, d\sigma_{g_{\alpha}}.
\end{align*}
Here $\nu$ denotes the outer unit normal derivative. Again, by a change of variable, we obtain using \eqref{strongconvua} 
\begin{align*}
\la\int_{\partial B(0,\delta\mu_{\alpha}/r_{\alpha})} \pr_{\nu}w_{\alpha}\, d\sigma_{g_{\alpha}}=\int_{\partial B(0,\delta)} \pr_{\nu}U\, d\sigma
\end{align*}
and \eqref{sing+har} implies 
\begin{equation*}
\la \int_{\partial B(0,1)} \pr_{\nu}w_{\alpha}\, d\sigma_{g_{\alpha}}=(2-n)w_{n-1}\mathscr{A}.
\end{equation*}
Hence, by letting $\alpha\to+\infty$ in  \eqref{po1}, we get for any $\delta>0$ small
\begin{align*}
(n-2)w_{n-1}\mathscr{A}+\int_{\partial B(0,\delta)} \pr_{\nu}U\, d\sigma=\int_{\R^{n}\setminus B(0,\delta)}\frac{U^{\crits-1}}{|X|^s}\, dX.
\end{align*}
Using the expression of $U$ given in \eqref{def:B0} and letting $\delta\to0$ gives us the claim, since
\begin{align*}
(n-2)w_{n-1}\mathscr{A}=\int_{\R^{n}}\frac{U^{\crits-1}}{|X|^s}\, dX
=(n-2)w_{n-1}\mathfrak{b}_{n,s}^{-\frac{n-2}{2-s}}.
\end{align*}
See for instance equation $(43)$ of \cite{HCA4}.
\hfill\qed \par
\medskip

Next, we obtain a sign on the harmonic part of $w_{\infty}$.
\begin{claim}\label{step:signg}
The harmonic function $\mathscr{H}$ in \eqref{sing+har} satisfies $\mathscr{H}(X)\geq0$ for all $X\in B(0,2)$.
\end{claim}
\noindent{\it Proof of the claim}:  From the coercivity of $\Delta_g+h_{\infty}$ it follows that for $\alpha\in \N$ is large enough, the operator $\Delta_{g}+h_{\alpha}$ is also coercive. Letting $G_{\alpha}>0$ denote the Green's function of  $\Delta_{g}+h_{\alpha}$. By \eqref{eqn:HSua} and Green's representation formula, we can write for all $x\in M$
\begin{align*}
u_{\alpha}(x)=\int_{M}G_{\alpha}(x,z)\frac{u_{\alpha}^{\crits-1}(z)}{d_{g}(\xo,z)^{s}}~dv_{g}(z),
\end{align*}
Then for any $X\in B(0,2)\setminus\{0\}$ we have
\begin{align*}
&\mathfrak{b}_{n,s}^{\frac{n-2}{2-s}}|X|^{n-2}\,w_{\alpha}(X)\ge~\mathfrak{b}_{n,s}^{\frac{n-2}{2-s}} |X|^{n-2}\\
&~\times\frac{r_{\alpha}^{n-2}}{\mu_{\alpha}^{\frac{n-2}{2}}}\int_{B(x_{0},\iota/2)}G_{\alpha}\(\exp_{x_{0}}(r_{\alpha}X),z\)\frac{u_{\alpha}^{2^{*}(s)-1}(z)}{d_{g}(x_{0},z)^{s}}~dv_{g}(z)\notag\\
&\ge\,\frac{\mathfrak{b}_{n,s}^{\frac{n-2}{2-s}}}{(n-2)\omega_{n-1}}\((n-2)\omega_{n-1}|r_{\alpha}X|^{n-2}G_{\alpha}(x_{0},\exp_{x_{0}}(r_{\alpha}X))\)\notag\\
&~\times\int_{B(0,\iota/2\mu_{\alpha})}\frac{G_{\alpha}\(\exp_{x_{0}}(r_{\alpha}X\),\exp_{x_{0}}(\mu_{\alpha}Z))}{G_{\alpha}(\exp_{x_{0}}(r_{\alpha}X),x_{0})}\,\frac{\widetilde{u}_{\alpha}^{2^{*}(s)-1}}{|Z|^{s}}~dv_{g}(\exp_{\xo}(\mu_{\alpha}Z)),
\end{align*}
where $\tua$ is defined as in \eqref{defuta} and  $\omega_{n-1}$ is the area of the unit sphere in $\R^n$.
Using the expansion of the Green's function (see appendix A of \cite{DHR} and Robert \cite{Rgreen}),  Lemma \ref{P2}  and Fatou’s lemma we obtain
\begin{align*}
&\int_{B(0,\iota/2\mu_{\alpha})}\frac{G_{\alpha}\(\exp_{x_{0}}(r_{\alpha}X\),\exp_{x_{0}}(\mu_{\alpha}Z))}{G_{\alpha}(\exp_{x_{0}}(r_{\alpha}X),x_{0})}\,\frac{\widetilde{u}_{\alpha}^{2^{*}(s)-1}}{|Z|^{s}}~dv_{g}(\exp_{\xo}(\mu_{\alpha}Z)).\\
&\hspace{1.5cm} \longrightarrow\bns^{\frac{2-n}{2-s}}(n-2)\omega_{n-1}\hbox{ as }\alpha\to+\infty,\\
&\hbox{ and } |r_{\alpha}X|^{n-2}G_{\alpha}(x_{0},\exp_{x_{0}}(r_{\alpha}X))\to \frac{1}{(n-2)\omega_{n-1}} \bb{ as } \alpha\to +\infty.
\end{align*}
Then as $\alpha\to+\infty$ one has
\begin{align*}
\mathfrak{b}_{n,s}^{\frac{n-2}{2-s}}|X|^{n-2}\,w_{\alpha}(X)\geq 1+\smallo(1),
\end{align*}
where $\smallo(1)\to 0$ as $\alpha\to +\infty$.
Thus the convergence in \eqref{sing+har} and equation \eqref{constantLambda} implies our claim that $\mathscr{H}(X)\geq0$ for all $X\in B(0,2)$.
\hfill\qed
\begin{claim}\label{s4}
The harmonic function $\mathscr{H}$ in \eqref{sing+har} satisfies $\mathscr{H}(0)\leq0$ in $B(0,2)$ for $n\geq4$ and $\mathscr{H}(0)=0$ for $n=3$.
\end{claim}
\noindent{\it Proof of the claim}:  Recall that, letting $\widehat{u}_{\alpha}(X) := u_{\alpha}(\exp_{\xo}(X))$,  $\ha(X)= h_{\alpha}(\exp_{\xo}( X))$ and $\widehat{g}(X):=( \exp^*_{\xo} g) (X)$, we have obtained the following Pohozaev's identity in \eqref{PoIden2}
\begin{eqnarray}\label{PoIden4}
\mathcal{R}_{\alpha}	&=&-\int_{B(0, r_\alpha)} \left( X^l \partial_l \uha+\frac{n-2}{2}\uha\right)\ha\uha\, dX\\
&&~-\int_{B(0, r_\alpha)} \left( X^l \partial_l \uha+\frac{n-2}{2}\uha \right)\left(\Delta_{\widehat{g}}\,\uha -\Delta_{\xi}\uha \right)\, dX,\notag
\end{eqnarray}
where $\mathcal{R}_{\alpha}$ is   the boundary term 
\begin{align*}
\mathcal{R}_{\alpha}:=&\int_{\partial B(0,r_\alpha)} \left( X,\nu\right) \left( \frac{|\nabla \uha |^2}{2} -\frac{1}{\crits}\frac{\uha^{\, \crits}}{|X|^s}\right)\,d\sigma(X)\notag\\
&~- \int_{\partial B(0,r_\alpha)}\left(X^l \partial_l \uha+\frac{n-2}{2}\uha \right) \partial_{\nu}\uha \, d\sigma(X).
\end{align*}
 Together with a change of variable and the convergence in \eqref{sing+har}, we have 
\begin{align*}
&\left( \frac{\ra}{\ma}\right)^{n-2}\mathcal{R}_{\alpha }=\int_{\partial B(0,1)} \left( X,\nu\right) \left( \frac{|\nabla w_{\alpha}|^2}{2} -\(\frac{\mu_{\alpha}}{r_{\alpha}}\)^{2-s}\frac{1}{\crits}\frac{w_{\alpha}^{\,\crits}}{|X|^s}\right)\, d\sigma(X)\\
&~- \int_{\partial B(0,1)}\left(X^l \partial_l w_{\alpha}+\frac{n-2}{2} w_{\alpha} \right) \partial_{\nu} w_{\alpha} \, d\sigma(X)\\
&=\mathscr{B}_{1}(w_{\infty})+\smallo(1) \bb{ as } \alpha\to +\infty,
\end{align*}
where for $0<\delta\leq1$
\begin{align*}
\mathscr{B}_{\delta}(w_{\infty}):=&\int_{\partial B(0,\delta)} \left( \frac{|\nabla w_{\infty} |^2}{2} -\left(  \partial_{\nu}w_{\infty}\right)^2 \right)\, d\sigma(X)\\
&~- \frac{n-2}{2}\int_{\pr B(0,\delta)} w_{\infty}\partial_{\nu}w_{\infty}\, d\sigma(X).
\end{align*}
Since $w_{\infty}$ is harmonic in $B(0,2)\setminus\{0\}$, using again the Pohozaev identity, it follows that $\mathscr{B}_{\delta}(w_{\infty})$ is independent of $\delta$. Then using the expression of $w_{\infty}$ in \eqref{sing+har} we obtain 
\begin{align*}
\mathscr{B}_{1}(w_{\infty})=\lim\limits_{\delta\to0}\mathscr{B}_{\delta}(w_{\infty})=\frac{(n-2)^2}{2}w_{n-1}\mathfrak{b}_{n,s}^{-\frac{n-2}{2-s}}\mathscr{H}(0).
\end{align*}
Thus 
\begin{equation}\label{Ralpha1}
\mathcal{R}_{\alpha}=\left( \frac{(n-2)^2}{2}\omega_{n-1}b_{n,s}^{-\frac{n-2}{2-s}}\mathscr{H}(0)+\smallo(1)\right)\frac{\ma^{n-2}}{\ra^{n-2}}.
\end{equation}

\noindent
 We have already calculated the other terms in the Pohozaev identity \eqref{PoIden4} in \eqref{calpha} and \eqref{dalpha}, giving us that, as $\alpha\to +\infty$ 
\begin{align}\label{eq3}
&-\int_{B(0, r_\alpha)} \left( X^l \partial_l \uha+\frac{n-2}{2}\uha\right)\ha\uha\, dX\notag\\
&~-\int_{B(0, r_\alpha)} \left( X^l \partial_l \uha+\frac{n-2}{2}\uha \right)  \left(\Delta_{\widehat{g}}\,\uha -\Delta_{\xi}\uha  \right) \, dX\nonumber\\
&=\left\{\begin{array}{ll}  
\bigO\left(r_{\alpha}\ma\right)  &\bb{ for } n=3\\\\
\ma^2\ln\left(\dfrac{r_{\alpha}}{\mu_{\alpha}}\right)\left(\varphi_\infty(x_{0})\,\omega_3\bn^{-\frac{4}{2-s}}+\smallo(1)\right)&\hbox{ for }n=4\\\\
\ma^2\(\varphi_\infty(x_{0}) \displaystyle\int_{\R^n} U^2 \, dX+\smallo(1)\)  &\hbox{ for } n\geq 5,
\end{array}\right.
\end{align}
Here $\varphi_{\infty}(x_{0})$ is defined as in \eqref{defpo}. Going back to \eqref{PoIden4} with \eqref{Ralpha1} and \eqref{eq3}, we get
\begin{align*}
-\frac{1}{2}\omega_3 \bnss^{-\frac{1}{2-s}}\mathscr{H}(0)&= 0   &\hbox{ if } n=3,\\
2\mathscr{H}(0)&= \po(\xo) \lim_{\alpha\to +\infty}\ra^2\ln\left( \frac{\ra}{\ma}\right)  &\hbox{ if } n=4,\\
\frac{(n-2)^2}{2}\omega_{n-1}\bns^{-\frac{n-2}{2-s}}\mathscr{H}(0)&=\left( \po(\xo)\int_{\R^n} U^{2}\, dX \right) \lim_{\alpha\to +\infty}\frac{\ra^{n-2}}{\ma^{n-4}}&\hbox{ if } n\geq 5.
\end{align*}
We get our claim \eqref{s4}, since $\lim\limits_{\alpha\to+\infty}r_{\alpha}=0$ and $\po(\xo)\leq 0$ when $n\geq 4$. \hfill\qed \par
\noindent
Step \eqref{s3} then follows from the maximum principle since $\mathscr{H}$ is harmonic.\hfill\qed 
\begin{step}\label{s5}
We have  $$\liminf\limits_{\alpha\to+\infty}r_{\alpha}>0.$$
\end{step}
\noindent{\emph{Proof of Step \ref{s4}} : }If $\lim\limits_{\alpha\to+\infty}r_{\alpha}=0$ up to a subsequence, there exists $(y_{\alpha})_{\alpha\in \N}$ in $M$ be such that $d_{g}(x_{0},y_{\alpha})=r_{\alpha}=\smallo(1)$ as $\alpha\to+\infty$ and
\begin{equation*}
|u_{\alpha}(y_{\alpha})-B_{\alpha}(y_{\alpha})|=\varepsilon B_{\alpha}(y_{\alpha}).
\end{equation*}
From claim \eqref{step:signg} we have  $u_{\alpha}(y_{\alpha})=(1+\varepsilon)B_{\alpha}(y_{\alpha})$. Let now $\Theta_{\alpha}:=r_{\alpha}^{-1}\exp_{x_{0}}^{-1}y_{\alpha}$. Then $|\Theta_{\alpha}|=1$ and 
$$w_{\alpha}(\Theta_{\alpha})=(1+\varepsilon)\(\(\dfrac{\mu_{\alpha}}{r_{\alpha}}\)^{2-s}+\mathfrak{b}_{n,s}\)^{\frac{2-n}{2-s}}.$$ 
Passing to the limit we obtain using step \eqref{s3}, that : $$[(n-s)(n-2)]^{\frac{n-2}{2-s}}=(1+\varepsilon)[(n-s)(n-2)]^{\frac{n-2}{2-s}},$$ a contradiction. This ends Step~\ref{s4}.\hfil\qed 
\smallskip

Finally, it is sufficient to take $\rho_\varepsilon := \liminf\limits_{\alpha \to +\infty} r_\alpha > 0$, which implies the claimed estimate~\eqref{es:main}. This ends the proof of Proposition~\ref{P3}.
\end{proof}
\noindent
Note that the blow-up sequence then \emph{almost} behaves like one bubble blow-up. Using Harnack inequality (Theorem 8.20 in \cite{GT}) we then obtain:
\begin{corollary}\label{co2}
Let $(M,g)$ be a closed and smooth Riemannian manifold of dimension $n\geq 3$, fix $\xo \in M$ and $s\in (0,2)$.  Let $(h_{\alpha})_{\alpha\in \N}$ be as in the statement of Proposition \ref{main:prop} and consider $(u_{\alpha})_{\alpha\in \N}$ satisfying  equation \eqref{eqn:HSua} such that $\lim \limits_{\alpha \to+\infty} \Vert u_{\alpha}\Vert_{L^\infty(M)} =+\infty$. Then there exists $C>0$ such that 
\begin{align}\label{global_{}control}
u_{\alpha}(x)\le C\,B_{\alpha}(x)\hbox{ for all  } x\in M.
\end{align}
\end{corollary}

\noindent
Also, as in Proposition 3.4 of \cite{HCA4} it then follows that a rescaling of the blow-up sequence behaves like the Green's function.
\begin{corollary}\label{lim:u:G}
Let $(M,g)$ be a closed and smooth Riemannian manifold of dimension $n\geq 3$, fix $\xo \in M$ and $s\in (0,2)$.  Let $(h_{\alpha})_{\alpha\in \N}$ be as in the statement of Proposition \ref{main:prop} and consider $(u_{\alpha})_{\alpha\in \N}$ satisfying  equation \eqref{eqn:HSua} such that $\lim \limits_{\alpha \to+\infty} \Vert u_{\alpha}\Vert_{L^\infty(M)} =+\infty$. Then there exists a dimensional constant $d_{n}>0$ such that 
\begin{equation}
\lim_{\alpha\to +\infty}\ma^{\frac{2-n}{2}}u_{\alpha}=d_{n}\,G_{x_0} \bb{ in } C^{1}_{loc}(M\backslash \{x_0\}),
\end{equation}
where $G_{x_0}$ is the Green's function of $\Delta_g+h_{\infty}$ on $M$ with  pole at $x_{0}$.
\end{corollary}
\medskip

\noindent{\bf Proof of the Proposition \ref{main:prop} : } It follows from Proposition~\ref{P3} that conditions \eqref{acontrol1} and \eqref{acontrol2} holds in lemma \eqref{lem:poho}, with $ \rho_{\epsilon} > 0$ for all $\eps>0$. Then, using Lemma~\ref{lem:poho}, we obtain condition $(i)$.
\begin{align*}
\varphi_{\infty}(\xo) = 0 \quad \text{for all } n \geq 4.
\end{align*}

\noindent 
The case $n=3$, where the ``mass'' of the operator $\Delta_{g}+h$ plays the crucial role, follows from  Step 4.13 of \cite{HCA4}. In fact,  the Pohozaev identity \eqref{PoIden2} and lemma \eqref{lim:u:G} as in Step 4.13 of \cite{HCA4} implies 
\begin{align}\label{PoIden3}
&\delta\int_{\partial B(0,\delta)} \left(\frac{|\nabla \widehat{G}_{x_0}|^{2}}{2}  +\widehat{h}_{\infty}\frac{\widehat{G}_{x_0}^{\,2}}{2}\right)\, d\sigma(X)\notag\\
&-\frac{1}{\delta}\int_{\partial B(0,\delta)}\left(\langle X, \nabla \widehat{G}_{x_0}\rangle^2 +\frac{1}{2}\langle X, \nabla \widehat{G}_{x_0}\rangle \widehat{G}_{x_0} \right) \, d\sigma(X)=\bigO\left(\delta \right)
\end{align}
for all $\delta>0$ sufficiently small. Here    
\begin{equation}\label{greenfunction}
\widehat{G}_{x_0}(x):=G(x_0,\exp_{x_0}(X))=\frac{1}{4\pi |X|}+\beta_{h_{\infty}}(x_{0},\exp_{x_0}(X)),
\end{equation}
where we have used \eqref{green0}.
Computing the terms in \eqref{PoIden3} and letting $\delta\to 0$ yields that the {\emph{mass}} $m_{h_{\infty}}(x_{0})=\beta_{h_{\infty}}(x_{0,}x_{0})=0$. For details see  Step 4.13 of \cite{HCA4}. This completes the proof of Proposition~\ref{main:prop}. 
\medskip


\bibliographystyle{amsplain}
\bibliography{biblio}
\end{document}